\documentclass[11pt]{amsart}

\usepackage{amsmath,amssymb,latexsym,soul,cite}
\usepackage{color,enumitem,graphicx}
\usepackage[colorlinks=true,urlcolor=blue,
citecolor=red,linkcolor=blue,linktocpage,pdfpagelabels,
bookmarksnumbered,bookmarksopen]{hyperref}
\usepackage[english]{babel}
\usepackage[T1]{fontenc}
\usepackage[utf8]{inputenc}
\usepackage[hyperpageref]{backref}

\usepackage[left=2.8cm,right=2.8cm,top=2.9cm,bottom=2.9cm]{geometry}

\numberwithin{equation}{section}

\newtheorem{thm}{Theorem}[section]
\newtheorem{lema}[thm]{Lemma}
\newtheorem{prop}[thm]{Proposition}
\newtheorem{cor}[thm]{Corollary}
\theoremstyle{remark}
\newtheorem{remark}[thm]{Remark}

\theoremstyle{definition}
\newtheorem{definition}[thm]{Definition}

\def \dist {\mathrm{dist}}
\def\supp{\mathop{\text{\normalfont supp}}}

\newcommand{\R}{{\mathbb R}}
\newcommand{\N}{{\mathbb N}}
\newcommand{\ve}{\varepsilon}
\def\cp{\operatorname {\text{cap}}}

\title[Compactness and dichotomy in nonlocal shape optimization]{Compactness and dichotomy in nonlocal shape optimization}

\author[E. Parini]{Enea Parini}
\author[A. Salort]{Ariel Salort}

\address[E. Parini]{Aix Marseille Univ, CNRS, Centrale Marseille, I2M, 39 Rue Frédéric Joliot Curie, 13453 Marseille CEDEX 13, France}
\email{enea.parini@univ-amu.fr} \urladdr{http://www.i2m.univ-amu.fr/perso/enea.parini}
\address[A. Salort]{Departamento de Matem\'atica, FCEN -- Universidad de Buenos Aires and IMAS -- CONICET, Buenos Aires, Argentina}
\email{asalort@dm.uba.ar}
\urladdr{http://mate.dm.uba.ar/~asalort}

\subjclass[2010]{}

\thanks{}

\begin{document}

\begin{abstract}
We prove a general result about the behaviour of minimizing sequences for nonlocal shape functionals satisfying suitable structural assumptions. Typical examples include functions of the eigenvalues of the fractional Laplacian under homogeneous Dirichlet boundary conditions. Exploiting a nonlocal version of Lions' concentration-compactness principle, we prove that either an optimal shape exists, or there exists a minimizing sequence consisting of two ``pieces'' whose mutual distance tends to infinity. Our work is inspired by similar results obtained by Bucur in the local case.
\end{abstract}

\maketitle

\section{Introduction}

A significant task in Shape Optimization consists in proving existence of minimizing sets, in a suitable class, for shape functionals of the kind
\[ \Omega \mapsto J(\Omega) = F(\lambda_1(\Omega),...,\lambda_m(\Omega)),\]
where $m \in \N^*$, $\Omega \subset \R^N$, and $\lambda_1(\Omega), ...,\lambda_m(\Omega)$ are eigenvalues of some differential operator. In the case of the Laplacian under Dirichlet boundary conditions, and $J(\Omega)=\lambda_k(\Omega)$, existence of optimal shapes among all   measurable  sets with prescribed Lebesgue measure has been a challenging open problem for a long time. Apart from the simpler cases $k=1$ and $k=2$, where the Faber-Krahn inequality implies that the optimal shape is a ball (for $k=1$) or the disjoint union of two equal balls (for $k=2$), for the general case existence in the class of quasi-open sets has been proven only recently by Bucur in \cite{bucurARMA} and by Mazzoleni and Pratelli in \cite{mazzolenipratelli} independently. It is still an open problem to identify the optimal shapes for $k \geq 3$, although numerical simulations support some conjectures.

When the differential operator under consideration is the fractional Laplacian, defined as
\[ (-\Delta)^s u(x) :=  C_{s,N}  \lim_{\varepsilon \to 0} \int_{\R^N \setminus B_\varepsilon(x)} \frac{u(x)-u(y)}{|x-y|^{N+2s}}\,dy,\]
where $s \in (0,1)$ and $C_{s,N}$ is a normalization constant, the situation is quite different. While the ball minimizes again the first eigenvalue under a volume constraint, the problem
\begin{equation} \label{minlambda2} \min \{\lambda_2(\Omega)\,|\,\Omega \subset \R^N,\,|\Omega|=c \}, \end{equation}
where $c > 0$, and $|\Omega|$ is the Lebesgue measure of $\Omega$, does not have a solution. Indeed, it was proven by Brasco and the first author \cite{brascoparini} that, for every admissible set $\Omega$,
\[ \lambda_2 (\Omega) > \lambda_1 (\widetilde{B}),\]
where $\widetilde{B}$ is a ball of volume $\frac{c}{2}$, and that a minimizing sequence $\{\Omega_n\}_{n \in \N}$ such that $\lambda_2(\Omega_n) \to \lambda_1(\widetilde{B})$  is given by the union of two disjoint balls of volume $\frac{c}{2}$, such that their mutual distance tends to infinity. This means that, in the nonlocal case, a general existence result as in \cite{bucurARMA} or \cite{mazzolenipratelli} can not hold true. On the other hand, if one restricts the minimization to quasi-open sets which are contained in a fixed open set $D \subset \R^N$, a generalization of the   existence  result by Buttazzo and Dal Maso \cite{buttazzodalmaso} holds true, as shown by Fern\'{a}ndez Bonder, Ritorto and the second author in \cite{fernandezbonderritortosalort}.

Inspired by the results obtained in \cite{bucur} by Bucur, in this paper we prove that, in the case of the fractional Laplacian, for a minimizing sequence only two situations can occur: \emph{compactness}, which implies, under some assumptions, existence of an optimal shape; or \emph{dichotomy}, which means that the sequence essentially behaves as the union of two disconnected sets, whose mutual distance tends to infinity, as in Problem \eqref{minlambda2}. To prove the result, we make use of a nonlocal version of the celebrated concentration-compactness principle of Lions \cite{lions}. Although some generalizations of Lions' result to the fractional case are stated in the literature, the proofs contained therein do not seem completely satisfactory, and therefore we prefer to provide our own proof. Our first main result reads as follows.   

\begin{thm} \label{ccp} 
Let $\{u_n\}_{n \in \N}$ be a bounded sequence in $H^s(\R^N)$ with $\int_{\R^N} |u_n|^2 \to \lambda$ for $n \to +\infty$. Then there exists a subsequence $\{n_k\}_{k \in \N}$ such that one of the following three cases occur:
\begin{enumerate}
 \item[\emph{(i)}] Compactness: there exists $\{y_k\}_{k \in \N} \subset \R^N$ such that
 \[ 
 \forall\varepsilon >0,\,\exists R < +\infty \text{ s.t. }\int_{y_k + B_R} |u_{n_k}|^2 \geq \lambda -\varepsilon. 
 \]
 \item[\emph{(ii)}] Vanishing: 
 \[ 
 \lim_{k \to +\infty} \sup_{y \in \R^N} \int_{y + B_R} |u_{n_k}|^2 = 0 \qquad \forall R>0.
 \]
 \item[\emph{(iii)}] Dichotomy: there exists $\alpha \in (0, \lambda)$, such that for all $\varepsilon>0$, there exist $k_0 \in \N$, $\{v_k\}_{k \in \N}$, $\{w_k\}_{k \in \N} \subset H^s(\R^N)$ such that, for $k \geq k_0$:
 \[ \|u_{n_k}-v_k-w_k\|_{L^2(\R^N)} \leq \delta(\varepsilon) \to 0 \qquad \text{for }\varepsilon \to 0;\]
 \[ \bigg| \int_{\R^N} |v_k|^2 - \alpha \bigg| \leq \varepsilon,\qquad \bigg| \int_{\R^N} |w_k|^2 - (\lambda-\alpha) \bigg| \leq \varepsilon;\]
 \[ \dist(\supp v_k,\supp w_k) \to +\infty \qquad \text{for }k \to +\infty;\]
 \begin{equation} \label{liminfconccomp} [u_{n_k}]^2_{H^s(\R^N)}-[v_k]^2_{H^s(\R^N)}-[w_k]^2_{H^s(\R^N)} \geq -2\varepsilon. \end{equation}
\end{enumerate}
\end{thm}

Then, we apply Theorem \ref{ccp} to the sequence of \emph{torsion functions} $w_{\Omega_n}$, where $\{\Omega_n\}_{n \in \N}$ is a minimizing sequence for the shape functional under consideration, which are defined as the weak solutions of the problems
\begin{equation} \label{eq.1}
\left\{\begin{array}{r c l l} (-\Delta)^s w_{\Omega_n} & = & 1 & \text{in }\Omega_n, \\ w_{\Omega_n} & = & 0 & \text{in }\R^N \setminus \Omega_n.\end{array}\right.
\end{equation}

In order to introduce our main result, we recall that a sequence $\{\Omega_n\}_{n\in\N}$ of $s$-quasi open sets of uniformly bounded Lebesgue measure is said to \emph{$\gamma$-converge} to the $s$-quasi open set $\Omega$ if the solutions $w_{\Omega_n}$ of \eqref{eq.1} strongly converge in $L^2(\R^N)$ to the solution $w_\Omega\in H^s_0(\Omega)$ of the problem
\[ \left\{\begin{array}{r c l l} (-\Delta)^s w_\Omega & = & 1 & \text{in }\Omega, \\ w_\Omega & = & 0 & \text{in }\R^N \setminus \Omega \end{array}\right.\]
(see Section \ref{sec.prel} for precise definitions of $s$-quasi open sets). Moreover, we say that a sequence $\{\Omega_n\}_{n\in\N}$ of $s$-quasi open sets of uniformly bounded Lebesgue measure \emph{weakly $\gamma$-converges} to the $s$-quasi open set $\Omega$ if the solutions $w_{\Omega_n}$ of \eqref{eq.1} converge weakly in $H^s(\R^N)$, and strongly in $L^2(\R^N)$, to a function $w \in H^s(\R^N)$ such that $\Omega = \{w > 0\}$. Finally, for a given $s$-quasi open set $\Omega\subset\R^N$  of finite measure, we denote by $R_\Omega$ the \emph{resolvent operator} of $(-\Delta)^s$, which is defined as the function $R_\Omega:L^2(\R^N)\to L^2(\R^N)$ such that $R_\Omega(f)=u$, where $u$ is the weak solution of
\[ \left\{\begin{array}{r c l l} (-\Delta)^s u & = & f & \text{in }\Omega, \\ u & = & 0 & \text{in }\R^N \setminus \Omega.\end{array}\right.\]

We can now state our second main result, whose proof follows the ideas of \cite{bucur}.

\begin{thm} \label{main}
Let $\{\Omega_n\}_{n \in \N}$ be a sequence of quasi-open sets of uniformly bounded measure. Then there exists a subsequence, still denoted by the same index, such that one of the following situations occurs:
\begin{enumerate}
 \item[(i)] Compactness: there exists a (possibly empty) quasi-open set $\Omega$, and a sequence $\{y_n\}_{n \in \N} \subset \R^N$, such that $y_n + \Omega_n$ weakly $\gamma$-converges to $\Omega$ as $n \to +\infty$.
 
 \item[(ii)] Dichotomy: there exists a sequence of subsets $\widetilde{\Omega}_n \subset \Omega_n$ such that \[ \|R_{\Omega_n}-R_{\widetilde{\Omega}_n}\|_{\mathcal{L}(L^2(\R^N))} \to 0, \quad  \quad \widetilde{\Omega}_n = \Omega_n^1 \cup \Omega_n^2,\] where $\text{dist}(\Omega_n^1,\Omega_n^2)\to +\infty$ and $\liminf_{n \to +\infty} |\Omega_n^i| > 0$ for $i=1,2$.
\end{enumerate}
\end{thm}

Theorem \ref{main} gives, as a consequence, an existence result for optimal shapes for minimization problems, when the shape functional satisfies some structural assumptions.

\begin{thm} \label{compactnessdichotomy}
Let 
\[ \mathcal{A}(\R^N) := \left\{ \Omega \subset \R^N\,|\,\Omega \ \ s\text{-quasi open} \right\} \]
and let $J:\mathcal{A}(\R^N) \to (-\infty,+ \infty]$ be a shape functional satisfying the following assumptions:
\begin{enumerate}
 \item[(i)] $J$ is lower semicontinuous with respect to $\gamma$-convergence;
 \item[(ii)] $J$ is decreasing with respect to set inclusion: if $\Omega_1$, $\Omega_2 \in \mathcal{A}(\R^N)$, $\Omega_1 \subset \Omega_2$, then $J(\Omega_2)\leq J(\Omega_1)$;
 \item[(iii)] $J$ is invariant by translations;
 \item[(iv)] $J$ is bounded from below.
\end{enumerate}
Let $c>0$, and define \begin{equation} \label{minimizingproblem} m:= \inf \{J(\Omega)\,|\, \Omega \in {\mathcal{A}(\R^N)}, |\Omega|=c\}. \end{equation}
Then, one of the following situations occurs:
\begin{enumerate}
 \item[(i)]  Existence of an optimal shape: there exists a $s$-quasi open set $\hat{\Omega} \in \mathcal{A}(\R^N)$ such that $|\hat{\Omega}|=c$ and $J(\hat{\Omega})=m$.
 \item[(ii)] Dichotomy: there exists a minimizing sequence $\{\Omega_n\}_{n \in \N}$ with $|\Omega_n|=c$ for every $n \in \N$, such that $\Omega_n = \Omega_n^1 \cup \Omega_n^2$, where $\Omega_n^1$, $\Omega_n^2$ are such that $\text{dist}(\Omega_n^1,\Omega_n^2) \to  +\infty$, $\liminf_{n \to  +\infty} |\Omega_n^i| > 0$ for $i=1,2$, and $J(\Omega_n) \to m$ as $n \to  +\infty$.
 \end{enumerate}
\end{thm}

Theorem \ref{compactnessdichotomy} applies in particular to spectral functionals of the kind
\[ 
J(\Omega):= F(\lambda_1(\Omega),...,\lambda_k(\Omega)),
\]
where $k \in \N$, $\lambda_j(\Omega)$ is the $j-$th eigenvalue of the Dirichlet fractional Laplacian, and $F: \R^k \to \R \cup \{ +\infty\}$ is a functional which is lower semicontinuous and nondecreasing in each variable.

\medskip

In the local case, existence of an optimal shape and the dichotomy situation can occur at the same time. Indeed, as we have pointed out, the classical Hong-Krahn-Szego inequality asserts that among all domains of fixed volume, the disjoint union of two equal balls has the smallest second eigenvalue. However, due to the nonlocal effects of the fractional Laplacian, the mutual position of two connected component has influence over the second eigenvalue, implying nonexistence of an optimal shape. Therefore it makes sense to ask whether existence of an optimal shape and dichotomy are two mutually exclusive situations in the nonlocal case. Up to our knowledge, this remains an open question.

\medskip

The manuscript is organized as follows. In section \ref{sec.prel} we introduce some preliminary definitions and notation. Section \ref{sec.ccp} is devoted to prove the concentration-compactness principle in the fractional setting. In section \ref{sec.convergencias} we define the notion of $\gamma$- and weak $\gamma$-convergence of sets as well as some  related  useful result, and finally in sections \ref{sec.main} and \ref{sec.compactnessdichotomy} we provide a proof of our main results.

\medskip

\textbf{Acknowledgements.}  The authors would like to express their gratitude to Lorenzo Brasco and Marco Squassina for useful discussions.  This work was started during a visit of A. S. to Aix-Marseille University in October 2017. The visit was supported by CONICET PIP 11220150100036CO.
A.S. wants to thank the first author for his hospitality which made the visit very enjoyable.

\section{Definitions and preliminary results} \label{sec.prel}

We begin this section with some definitions.

\subsection{Fractional Sobolev spaces and $s$-capacity of sets}

For $s \in (0,1)$, the fractional Sobolev space $H^s(\R^N)$ is defined as
\[ H^s(\R^N):= \left\{ u \in L^2(\R^N)\,|\,[u]_{H^s(\R^N)} <  +\infty \right\}, \]
endowed with the norm $\|\cdot\|_{H^s(\R^N)}$ defined by
\[ \|u\|_{H^s(\R^N)} := \left( \|u\|_{L^2(\R^N)}^2 + [u]_{H^s(\R^N)}^2\right)^{\frac{1}{2}},\]
 where $[\cdot]_{H^s(\R^N)}$ is the \emph{Gagliardo seminorm} defined as
\[ [u]_{H^s(\R^N)} := \left(\int_{\R^N} \int_{\R^N} \frac{|u(x)-u(y)|^2}{|x-y|^{N+2s}}\,dx\,dy \right)^{\frac{1}{2}}.\]
The Gagliardo seminorm of a function $u \in H^s(\R^N)$ can also be expressed in terms of its Fourier transform $\mathcal{F}u$ as
\[ [u]_{H^s(\R^N)}^2 = \frac{2}{C_{s,N}} \int_{\R^N} |\xi|^{2s} |\mathcal{F}u(\xi)|^2\,d\xi,\]
where $C_{s,N}$ is the normalization constant in the definition of $(-\Delta)^s$, given by
\[ C_{s,N} = \left(\int_{\R^N} \frac{1-\cos{\zeta_1}}{|\zeta|^{N+2s}}\,d\zeta\right)^{-1}\]
(see \cite[Proposition 3.4]{dinezzapalatuccivaldinoci}).
Given a measurable set $\Omega\subset \R^N$, for any $s \in (0,1)$ we define the \emph{$s$-capacity} of $\Omega$ as
$$
\cp_s(\Omega)= \inf \left\{ [u]_{H^s(\R^N)}^2 \colon u\in H^s(\R^N), u\geq 1 \text{ a.e. on a neighborhood of } \Omega  \right\}.
$$
We say that a property holds \emph{$s$-quasi everywhere} if it holds up to a set of null $s$-capacity. A measurable subset $\Omega\subset \R^N$ is a {\em $s$-quasi open} set if there exists   a decreasing sequence $\{\omega_n\}_{n\in  \N}$ of open subsets of $\R^N$ such that $\cp_s(\omega_n) \to 0$, as $n\to+\infty$, and $\Omega \cup \omega_n$ is open.

A function $u\in H^s(\R^N)$ is said to be \emph{$s$-quasi continuous} if for every $\ve>0$ there exists an open set $G\subset \R^N$  such that $\cp_s(G)<\ve$ and $u|_{\R^N\setminus G}$ is continuous. It is well-known that $\cp_s$ is a Choquet capacity on $\R^N$ \cite[Section 2.2]{AH} and for every $u\in H^s(\R^N)$ there exists a unique $s$-quasi continuous function $\tilde u:\R^N\to \R$ such that $\tilde u = u$ $s$-quasi everywhere on $\R^N$. Therefore we will always consider, without loss of generality, that a function $u \in H^s(\R^N)$ coincides with its $s$-quasi continuous representative.  If $u:\R^N \to \R$ is $s$-quasi continuous, then every superlevel set $\{u > t\}$ is $s$-quasi open.

For a generic measurable set $\Omega\subset \R^N$, we define the fractional Sobolev space $H^s_0(\Omega)$ as 
$$
H^s_0(\Omega)=\{ u\in H^s(\R^N) : u=0 \quad s\text{-q.e. on } \R^N\setminus \Omega\}.
$$

The following Poincar\'e's inequality holds for measurable sets of finite measure.
\begin{prop} \label{poincare}
Let $\Omega \subset \R^N$ be a measurable set of finite Lebesgue measure. Then, there exists a constant $C=C(s, |\Omega|)>0$ such that, for every $u \in H^s_0(\Omega)$,
\[ \|u\|_{L^2(\Omega)}\leq C [u]_{H^s(\R^N)}.\]
\end{prop}
\begin{proof}
Let $u$ be a function in $H^s_0(\Omega)$ and consider the ball $\Omega^*$  such that $|\Omega^*|=|\Omega|$. Let $v:=|u|^*$ be the Schwarz symmetrization of $|u|$, as defined in \cite[Definition 1.3.1]{kesavan}. By \cite[Theorem 9.2]{almgrenlieb}, $v \in H^s_0(\Omega^*)$, and
\[[v]_{H^s(\R^N)}\leq [|u|]_{H^s(\R^N)} \leq [u]_{H^s(\R^N)}.\]
By \cite[Lemma 2.4]{brascolindgrenparini}, there exists $C=C(s,|\Omega|)>0$ such that
\[ \|v\|_{L^2(\Omega^*)} \leq C[v]_{H^s(\R^N)}. \]
Since symmetrization preserve the $L^2$-norm,
\[ \|u\|_{L^2(\Omega)} = \|v\|_{L^2(\Omega^*)} \leq C[v]_{H^s(\R^N)} \leq C[u]_{H^s(\R^N)}, \]
and the claim follows.
\end{proof}

The previous proposition leads to a useful compactness result.
\begin{prop}
Let $\Omega \subset \R^N$ be a measurable set of finite Lebesgue measure. Then, for every bounded sequence $\{u_n\}_{n \in \N}$ in $H^s_0(\Omega)$, there exists a subsequence $\{u_{n_k}\}_{k \in \N}$ and a function $u \in H^s_0(\Omega)$ such that $u_{n_k} \to u$ in $L^2(\Omega)$.
\end{prop}
\begin{proof}
The proof can be performed as in \cite[Theorem 2.7]{brascolindgrenparini}, using the Poincaré inequality stated in Proposition \ref{poincare}. 
\end{proof}

Given an $s$-quasi open set $\Omega$ of finite Lebesgue measure and $f\in L^2(\R^N)$ we denote by $R_\Omega$ the resolvent operator of the fractional Laplacian with Dirichlet boundary conditions, that is, $R_\Omega: L^2(\R^N) \to L^2(\R^N)$ and $R_\Omega(f)=u$, where $u$ is the weak solution of 
\begin{equation} \label{eqA}
\left\{\begin{array}{rcll}
(-\Delta)^s u &=& f & \text{in }\Omega,\\
u & = & 0 & \text{in }\R^N \setminus \Omega.\end{array}\right.
\end{equation}
In particular, $w_\Omega=R_\Omega(1)$. It is easy to check that  $R_{\Omega}$  defines a continuous compact, self-adjoint linear operator from $L^2(\R^N)$ in itself. We denote by $\|\cdot\|_{\mathcal{L}(L^2(\R^N))}$ the corresponding operator norm. 
Given an $s$-quasi open set $\Omega$, we say that $\lambda$ is an \emph{eigenvalue} of the fractional Laplacian if there exists a nontrivial function $u \in H^s_0(\Omega)$, called \emph{eigenfunction}, which is a weak solution of
\begin{equation} \label{eqB}
\left\{\begin{array}{r c l l}
(-\Delta)^s u &=& \lambda u & \text{in }\Omega,\\
u & = & 0 & \text{in }\R^N \setminus \Omega.
\end{array}\right.
\end{equation}
According to Courant-Fischer's min-max principle, for every $s$-quasi-open set $\Omega\subset \R^N$ of finite Lebesgue measure there exists a sequence $\{\lambda_k(\Omega)\}_{k \in \N}$ of eigenvalues of the fractional Laplacian, satisfying
\[ 0 < \lambda_1(\Omega) < \lambda_2(\Omega) \leq  \dots \leq \lambda_k(\Omega) \to   +\infty \qquad \text{as }k \to  +\infty.\]
 The first eigenvalue $\lambda_1(\Omega)$ is characterized as
\[ \lambda_1(\Omega) = \inf_{u \in H^s_0(\Omega)\setminus \{0\}} \frac{[u]_{H^s(\R^N)}^2}{\|u\|_{L^2(\R^N)}^2}\]
and the associated first eigenfunction is unique (up to multiplicative constant) and strictly positive (or negative) in $\Omega$.

\medskip

Eigenfunctions satisfy the following regularity property.
\begin{prop}
Let $\Omega \subset \R^N$ be a quasi-open set of finite Lebesgue measure, and let $u \in H^s_0(\Omega)$ be an eigenfunction of the fractional Laplacian. Then, $u \in L^\infty(\Omega)$. 
\end{prop}
\begin{proof}
The proof can be performed as in \cite[Theorem 3.2]{franzinapalatucci} taking into account Theorems 6.5 and 6.9 from \cite{dinezzapalatuccivaldinoci}.
\end{proof}

\section{The concentration-compactness principle} \label{sec.ccp}

In this section we prove Theorem \ref{ccp}.

\begin{proof}[Proof of Theorem \ref{ccp}]
All the assertions of this theorem, with exception of \eqref{liminfconccomp}, follow from the classical concentration-compactness lemma \cite[Lemma 1.1]{lions}. To prove \eqref{liminfconccomp}, we suitably modify \cite[Lemma III.1]{lions}. Let $\varepsilon>0$, and let $R_0>0$ be chosen as in \cite[Lemma III.1]{lions}. Let us define two cut-off functions $\varphi, \psi \in C^\infty_c(\R^N)$ satisfying $0 \leq \varphi,\,\psi \leq 1$, $\varphi \equiv 1$ on $B_1$, $\varphi \equiv 0$ on $\R^N \setminus B_2$ and $\psi \equiv 0$ on $B_1$, $\psi \equiv 1$ on $\R^N \setminus B_2$. Denote by $\varphi_R$, $\psi_R$ the functions defined by
\begin{equation} \label{psi.varphi}
\varphi_R(x):=\varphi\left(\frac{x}{R}\right),\qquad \psi_R(x):=\psi\left(\frac{x}{R}\right).
\end{equation}
For any function $u \in H^s(\R^N)$ with $[u]_{H^s(\R^N)} \leq M$ we have
\begin{align*} 
& \int_{\R^N}\int_{\R^N} \frac{|\varphi_R(x)u(x)-\varphi_R(y)u(y)|^2}{|x-y|^{N+2s}}\,dx\,dy \\ 
& = \int_{\R^N}\int_{\R^N} \frac{|\varphi_R(x)u(x)+\varphi_R(x)u(y)-\varphi_R(x)u(y)-\varphi_R(y)u(y)|^2}{|x-y|^{N+2s}}\,dx\,dy \\ 
& = \int_{\R^N}\int_{\R^N} |\varphi_R(x)|^2 \frac{ |u(x)-u(y)|^2}{|x-y|^{N+2s}}\,dx\,dy  + \int_{\R^N}\int_{\R^N} |u(y)|^2\frac{|\varphi_R(x)-\varphi_R(y)|^2}{|x-y|^{N+2s}}\,dx\,dy \\ 
&\quad  + 2\int_{\R^N}\int_{\R^N} \frac{\varphi_R(x)u(y)[\varphi_R(x)-\varphi_R(y)][u(x)-u(y)]}{|x-y|^{N+2s}}\,dx\,dy.
\end{align*}
By the computations in \cite[Lemma A.2]{BSY}, it is possible to estimate
\[ \int_{\R^N}\int_{\R^N} |u(y)|^2\frac{|\varphi_R(x)-\varphi_R(y)|^2}{|x-y|^{N+2s}}\,dx\,dy \leq \frac{C}{R^{2s}},\]
where $C$ only depends on $\|\nabla \varphi\|_\infty$ and $\|u\|_{L^2(\R^N)}$.

Moreover, the Cauchy-Schwarz inequality together with the last inequality gives that
\begin{align*} & \int_{\R^N}\int_{\R^N} \frac{\varphi_R(x)u(y)[\varphi_R(x)-\varphi_R(y)][u(x)-u(y)]}{|x-y|^{N+2s}}\,dx\,dy \\ & \leq \left(\int_{\R^N}\int_{\R^N} \frac{|u(y)|^2 |\varphi_R(x)-\varphi_R(y)|^2}{|x-y|^{N+2s}}\,dx\,dy \right)^\frac{1}{2}\left(\int_{\R^N}\int_{\R^N} \frac{|\varphi_R(x)|^2|u(x)-u(y)|^2}{|x-y|^{N+2s}}\,dx\,dy \right)^\frac{1}{2} \\ & \leq \left(\int_{\R^N}\int_{\R^N} \frac{|u(y)|^2 |\varphi_R(x)-\varphi_R(y)|^2}{|x-y|^{N+2s}}\,dx\,dy \right)^\frac{1}{2}\left(\int_{\R^N}\int_{\R^N} \frac{|u(x)-u(y)|^2}{|x-y|^{N+2s}}\,dx\,dy \right)^\frac{1}{2} \\ & \leq \frac{C}{R^s},\end{align*}
where $C$ only depends on $\|\nabla \varphi\|_\infty$, $\|u\|_{L^2(\R^N)}$, and $[u]_{H^s(\R^N)}$. 

Similar computations hold true for the quantity
\[ \int_{\R^N}\int_{\R^N} \frac{|\psi_R(x)u(x)-\psi_R(y)u(y)|^2}{|x-y|^{N+2s}}\,dx\,dy.\]
Therefore it is possible to choose $R_1 \geq R_0$ such that, for $R \geq R_1$, and for every $n \in \N$,
\[ \bigg|\int_{\R^N}\int_{\R^N} \frac{|\varphi_R(x)u_n(x)-\varphi_R(y)u_n(y)|^2}{|x-y|^{N+2s}}\,dx\,dy - \int_{\R^N}\int_{\R^N} \frac{|\varphi_R(x)|^2 |u_n(x)-u_n(y)|^2}{|x-y|^{N+2s}}\,dx\,dy\bigg| \leq \varepsilon,\]
\[ \bigg|\int_{\R^N}\int_{\R^N} \frac{|\psi_R(x)u_n(x)-\psi_R(y)u_n(y)|^2}{|x-y|^{N+2s}}\,dx\,dy - \int_{\R^N}\int_{\R^N} \frac{|\psi_R(x)|^2 |u_n(x)-u_n(y)|^2}{|x-y|^{N+2s}}\,dx\,dy\bigg| \leq \varepsilon.\]
The claim follows defining 
\[ 
v_k(x) = \varphi_{R_1}(x-y_k)u_{n_k}(x), \qquad w_k(x) = \psi_{R_k}(x-y_k)u_{n_k}(x),
\]
where $y_k$ and $R_k \to  +\infty$  are defined as in \cite[pp 136-137]{lions}   and observing that
\begin{align*} & \int_{\R^N}\int_{\R^N} \frac{|u_{n_k}(x)-u_{n_k}(y)|^2}{|x-y|^{N+2s}}\,dx\,dy \\ & \geq \int_{\R^N}\int_{\R^N} \frac{|\varphi_{R_1}(x)|^2 |u_{n_k}(x)-u_{n_k}(y)|^2}{|x-y|^{N+2s}}\,dx\,dy + \int_{\R^N}\int_{\R^N} \frac{|\psi_{R_k}(x)|^2 |u_{n_k}(x)-u_{n_k}(y)|^2}{|x-y|^{N+2s}}\,dx\,dy
\end{align*}
since $\varphi_{R_1}$ and $\psi_{R_k}$ have disjoint support for $k$ big enough, and therefore
\[ |\varphi_{R_1}(x)|^2 + |\psi_{R_k}(x)|^2 \leq 1 \qquad \text{for every } x \in \R^N. \]
\end{proof}

\begin{cor}
In the dichotomy case, it is possible to find sequences $\{u_k^{(1)}\}_{k \in \N}$, $\{u_k^{(2)}\}_{k \in \N} \subset H^s(\R^N)$ such that
\[
\|u_{n_k}-u_k^{(1)}-u_k^{(2)}\|_{L^2(\R^N)}  \to 0 \qquad \text{for }k \to +\infty;
\]
\[ 
\int_{\R^N} |u_k^{(1)}|^2  \to \alpha ,\qquad \int_{\R^N} |u_k^{(2)}|^2 \to \lambda-\alpha \qquad \text{for }k \to +\infty;
\]
\[ 
\dist(\supp u_k^{(1)},\supp u_k^{(2)}) \to +\infty \qquad \text{for }k \to +\infty;
\]
\begin{equation} \label{liminfconccomp2} 
\liminf_{k \to +\infty} \left( [u_{n_k}]^2_{H^s(\R^N)} - [u_k^{(1)}]^2_{H^s(\R^N)}-[u_k^{(2)}]^2_{H^s(\R^N)}\right) \geq 0. \end{equation} 
\end{cor}

\section{$\gamma$-convergence of sets} \label{sec.convergencias}
In this section we introduce the notions of $\gamma$-convergence and weak $\gamma$-convergence of sets, and we prove some useful results leading to our main theorem. 
\begin{prop} \label{dinezzaweakconvergence}
Let $\{u_n\}_{n \in \N}$ be a sequence in $H^s(\R^N)$ such that $u_n \rightharpoonup u$ weakly in $H^s(\R^N)$ as $n \to  +\infty$. Then, for every function $\varphi \in W^{1,\infty}(\R^N)$, it holds that $\varphi u_n \in H^s(\R^N)$ for every $n \in \N$, and $\varphi u_n \rightharpoonup \varphi u$ weakly in $H^s(\R^N)$ as $n \to  +\infty$.
\end{prop}
\begin{proof}
The sequence $\{u_n\}_{n \in \N}$ is uniformly bounded in $H^s(\R^N)$. Moreover, since the embedding $H^s(B_r) \hookrightarrow L^2(B_r)$ is compact for every $r>0$, it follows that $u_n \to u$ strongly in $L^2(B_r)$ for every $r>0$. Arguing as in \cite[Lemma 5.3]{dinezzapalatuccivaldinoci}, we have that the sequence $\{\varphi u_n\}_{n \in \N}$ is also bounded in $H^s(\R^N)$. Therefore, every subsequence $\{\varphi u_{n_k}\}$ admits a subsequence $\{\varphi u_{n_{k_j}}\}$ which converges weakly in $H^s(\R^N)$, and almost everywhere in $\R^N$, to some $v \in H^s(\R^N)$. But $u_{n_{k_j}}$ must converge to $u$ almost everywhere in $\R^N$. Therefore, $\varphi u_{n_{k_j}} \to \varphi u$ a.e. in $\R^N$, and thus $v = \varphi u$. Hence all the sequence $\varphi u_n$ converges weakly in $H^s(\R^N)$ to $\varphi u$.
\end{proof}

\subsection{$\gamma$-convergence and continuity of the spectrum}
We prove that $\gamma$-convergence of $s$-quasi open sets implies the convergence of their resolvent operators in the $\mathcal{L}(L^2(\R^N))$ norm. In particular we obtain continuity of the spectrum with respect to the $\gamma$-convergence. 
\begin{definition} \label{definitiongammaconvergence}
Let $\{\Omega_n\}_{n \in \N}$ be a sequence of $s$-quasi open sets such that $|\Omega_n| \leq c$ for every $n \in \N$. We say that $\{\Omega_n\}_{n \in \N}$ \emph{$\gamma$-converges} to the $s$-quasi open set $\Omega$ if the solutions $w_{\Omega_n} \in H^s_0(\Omega_n)$ of the problems
\begin{equation} \label{eq.wn}
\left\{\begin{array}{r c l l} (-\Delta)^s w_{\Omega_n} & = & 1 & \text{in }\Omega_n, \\ w_{\Omega_n} & = & 0 & \text{in }\R^N \setminus \Omega_n,\end{array}\right.
\end{equation}
strongly converge in $L^2(\R^N)$ to the solution $w_\Omega \in H^s_0(\Omega)$ of the problem
\[ \left\{\begin{array}{r c l l} (-\Delta)^s w_\Omega & = & 1 & \text{in }\Omega, \\ w_\Omega & = & 0 & \text{in }\R^N \setminus \Omega. \end{array}\right.\]
\end{definition}

\begin{remark} \label{l2impliesHs}
We observe that, if $\{\Omega_n\}_{n\in \N}$ are $s$-quasi open sets, with $|\Omega_n|\leq c$, which $\gamma$-converge to $\Omega$, then $w_{\Omega_n} \to w_\Omega$ strongly in $H^s(\R^N)$. Indeed, by Propositions \ref{semicontinuityvolume} and \ref{strongimpliesweak}, one has $|\Omega|\leq c$. Therefore
 \begin{align*} \int_{\Omega_n} w_{\Omega_n} - \int_{\Omega} w_{\Omega}  & \leq \int_{\Omega_n\setminus \Omega} w_{\Omega_n} + \int_{\Omega_n \cap \Omega} |w_{\Omega_n} - w_{\Omega}| + \int_{\Omega \setminus \Omega_n} w_{\Omega} \\ &\leq \int_{\Omega_n \cup \Omega} |w_{\Omega_n} - w_\Omega| \leq  (2c)^{\frac{1}{2}}  \|w_{\Omega_n}-w_\Omega\|_{L^2(\R^N)}\end{align*}
 and therefore
 \[ \lim_{n \to +\infty} \int_{\Omega_n} w_{\Omega_n} = \int_{\Omega} w_\Omega.\]
Passing to the limit in the weak formulation, we obtain
\[ [w_{\Omega_n}]_{H^s(\R^N)}^2 = \int_{\Omega_n} w_{\Omega_n} \to \int_{\Omega} w_\Omega = [w_\Omega]_{H^s(\R^N)}^2\]
and therefore, by reflexivity of $H^s(\R^N)$, $w_{\Omega_n} \to w_\Omega$ strongly in $H^s(\R^N)$.  
\end{remark}

\begin{prop} \label{uniformbehaviour}
Let $\{\Omega_n\}_{n \in \N}$ be a sequence of $s$-quasi open sets of uniformly bounded measure, which $\gamma$-converges to the $s$-quasi open set $\Omega$. Let $\{u_n\}_{n \in \N}$ be a sequence in $H^s(\R^N)$ such that $u_n \in H^s_0(\Omega_n)$ for every $n \in \N$, and $u_n \rightharpoonup u$ weakly in $H^s(\R^N)$. Then, $u_n \to u$ strongly in $L^2(\R^N)$. 
\end{prop}
\begin{proof}
The proof goes as in \cite[Theorem 2.1]{bucur}. Denoting by $\mathcal{F}u_n$, $\mathcal{F}u$ the Fourier transforms of $u_n$ and $u$ respectively, for $R>0$ we have that
\begin{align*} \|u_n - u\|_{L^2(\R^N)}^2 & = \int_{\R^N} |\mathcal{F}u_n(\xi)-\mathcal{F}u(\xi)|^2\,d\xi \\ & = \int_{|\xi|\geq R} (1+|\xi|^{2s})^{-1}(1+|\xi|^{2s})|\mathcal{F}u_n(\xi)-\mathcal{F}u(\xi)|^2\,d\xi +  \int_{|\xi|<R} |\mathcal{F}u_n(\xi)-\mathcal{F}u(\xi)|^2\,d\xi \\ & \leq \frac{C_{s,N}}{1+R^{2s}}\|u_n - u\|_{H^s(\R^N)}^2 + \int_{|\xi|<R} |\mathcal{F}u_n(\xi)-\mathcal{F}u(\xi)|^2\,d\xi,\end{align*}
where the constant $C_{s,N}$ is the equivalence norm constant given \cite[Proposition 3.4]{dinezzapalatuccivaldinoci}. Let $\varepsilon >0$ be fixed. Since $\{u_n\}_{n \in \N}$ is bounded in $H^s(\R^N)$, there exists $R>0$ such that, for every $n \in \N$,
\[ \frac{C_{s,N}}{1+R^{2s}}\|u_n - u\|_{H^s(\R^N)}^2 < \frac{\varepsilon}{2}.\]
It remains to prove that 
\[ \int_{|\xi|<R} |\mathcal{F}u_n(\xi)-\mathcal{F}u(\xi)|^2\,d\xi \to 0\]
as $n \to  +\infty$. For $\xi \in B_R$, define the complex-valued function $g_\xi:\R^N \to \mathbb{C}$ as $g_\xi(x) = e^{2\pi i \langle x,\xi\rangle}$. By Proposition \ref{dinezzaweakconvergence} applied to the real and imaginary parts of $g_\xi$, it holds that $ug_\xi \in H^s_0(\Omega;\mathbb{C})$ and $u_n g_\xi \in H^s_0(\Omega_n;\mathbb{C})$ for every $n \in \N$, and $u_n g_\xi \rightharpoonup u g_\xi$ weakly in $H^s(\R^N;\mathbb{C})$ as $n \to  +\infty$. 

Let $w_{\Omega_n}\in H^s_0(\Omega_n)$ be the solution of \eqref{eq.wn}. Testing this equation with $u_n g_\xi$, we obtain
\[ \int_{\R^N} \int_{\R^N} \frac{(w_{\Omega_n}(x)-w_{\Omega_n}(y))(u_n(x)g_\xi(x)-u_n(y)g_\xi(y))}{|x-y|^{N+2s}}\,dx\,dy = \int_{\R^N} u_n(x)g_\xi(x)\,dx.\]
Letting $n \to  +\infty$ and observing that $w_{\Omega_n} \to w$ strongly in $H^s(\R^N)$   by Remark \ref{l2impliesHs}, we obtain
\[ \int_{\R^N} u_n(x)g_\xi(x)\,dx \to \int_{\R^N} u(x)g_\xi(x)\,dx\]
as $n \to +\infty$. Observing that
\[ \mathcal{F}u_n(\xi) = \int_{\Omega_n} u_n(x)g_\xi(x)\,dx\]
and
\[ \mathcal{F}u(\xi) = \int_{\Omega} u(x)g_\xi(x)\,dx,\]
we have $|\mathcal{F}u_n(\xi)-\mathcal{F}u(\xi)| \to 0$ as $n \to  +\infty$. Moreover,
\[ |\mathcal{F}u_n(\xi)| \leq \int_{\Omega_n} |u_n(x)|\,dx \leq |\Omega_n|^{\frac{1}{2}} \|u_n\|_{L^2(\R^N)},\]
and a similar relation holds for $\mathcal{F}u$. Therefore, $\mathcal{F}u_n$ and $\mathcal{F}u$ are uniformly bounded in $L^\infty$.
Applying Lebesgue's dominated convergence Theorem we get
\[\int_{|\xi|<R} |\mathcal{F}u_n(\xi)-\mathcal{F}u(\xi)|^2\,d\xi \to 0\]
and hence the claim.
\end{proof}

\begin{cor} \label{corollaryuniformbehaviour}
Let $\{\Omega_n\}_{n \in \N}$ be a sequence of $s$-quasi open sets such that $|\Omega_n| \leq c$ for every $n \in \N$. Suppose that $\{\Omega_n\}_{n \in \N}$ $\gamma$-converges to the $s$-quasi-open set $\Omega$. Then, for every sequence $f_n \in L^2(\Omega_n)$ converging weakly in $L^2(\R^N)$ to $f \in L^2(\Omega)$, the solutions $u_n \in H^s(\R^N)$ of the problems
\[ \left\{\begin{array}{r c l l} (-\Delta)^s u_n & = & f_n & \text{in }\Omega_n, \\ u_n & = & 0 & \text{in }\R^N \setminus \Omega_n,\end{array}\right.\]
strongly converge in $L^2(\R^N)$ to the solution $u \in H^s(\R^N)$ of the problem
\[ \left\{\begin{array}{r c l l} (-\Delta)^s u & = & f & \text{in }\Omega, \\ u & = & 0 & \text{in }\R^N \setminus \Omega. \end{array}\right.\]
\end{cor}
\begin{proof}
Exploiting the weak form of the equations, it is straightforward to see that $u_n \rightharpoonup u$ weakly in $H^s(\R^N)$. By Proposition \ref{uniformbehaviour}, $u_n \to u$ strongly in $L^2(\R^N)$. 
\end{proof}

\begin{prop} \label{uniformconvergenceresolvent}
Let $\{\Omega_n\}_{n \in \N}$ be a sequence of $s$-quasi open sets such that $|\Omega_n| \leq c$ for every $n \in \N$. Suppose that $\{\Omega_n\}_{n \in \N}$ $\gamma$-converges to the $s$-quasi open set $\Omega$. Then, the resolvents $R_{\Omega_n}$ converge to $R_\Omega$ in $\mathcal{L}(L^2(\R^N))$. In particular, for every $k \geq 1$,
\[ \lambda_k(\Omega_n) \to \lambda_k(\Omega) \qquad \text{as }n \to +\infty.\]
\end{prop}

\begin{proof}
We have to show that
\[
	\lim_{n\to+\infty} \sup \left\{ \|R_{\Omega_n}(f) - R_\Omega(f)\|_{L^2(\R^N)}\,\big|\,f \in L^2(\R^N),\,\|f\|_{L^2(\R^N)}\leq 1\right\} =0.
\]
It is equivalent to prove that, for every sequence $\{f_n\}_{n\in\N}$ such that $\|f_n\|_{L^2(\R^N)}= 1$, the following limit holds
\[
	\lim_{n\to +\infty} \|R_{\Omega_n}(f_n) - R_\Omega(f_n)\|_{L^2(\R^N)} =0.
\]
Let $\{f_n\}_{n\in\N}$ be such a sequence. Without loss of generality, we can suppose that there exists $f\in L^2(\R^N)$ such that $f_n \rightharpoonup f$ in $L^2(\R^N)$. By the triangular inequality we get
\begin{align*}
\limsup_{n \to +\infty} \|R_{\Omega_n}(f_n) &- R_\Omega(f_n)\|_{L^2(\R^N)} \leq\\
&\limsup_{n \to +\infty} \|R_{\Omega_n}(f_n) - R_\Omega(f)\|_{L^2(\R^N)} +
\limsup_{n \to +\infty} \|R_{\Omega}(f_n) - R_\Omega(f)\|_{L^2(\R^N)}.
\end{align*}
The first term in the previous inequality is equal to zero by Corollary \ref{corollaryuniformbehaviour}, while the second term is also zero since the injection $H^s_0(\Omega) \to L^2(\Omega)$ is compact due to Proposition  \ref{poincare}. By \cite[Lemma XI.9.5]{dunfordschwartz}, we have, for every $k \geq 1$,
\begin{equation} \label{estimatedunford} \bigg| \frac{1}{\lambda_k(\Omega_n)}-\frac{1}{\lambda_k(\Omega)}\bigg| \leq \|R_{\Omega_n}-R_{\Omega}\|_{\mathcal{L}(L^2(\R^N))}\end{equation}
and hence
\[ \lambda_k(\Omega_n) \to \lambda_k(\Omega) \qquad \text{as }n \to +\infty,\]
concluding the proof.
\end{proof}

\begin{remark} \label{remarkemptyset}
When $\Omega=\emptyset$ quasi-everywhere, by definition $H^s_0(\Omega)=\{0\}$, $R_{\Omega}$ is the null operator, and formally $\lambda_k(\Omega)= +\infty$ for every $k \geq 1$. In this case, \eqref{estimatedunford} becomes
\begin{equation} \label{estimatedunford2} 0 \leq \frac{1}{\lambda_k(\Omega_n)}  \leq \|R_{\Omega_n}\|_{\mathcal{L}(L^2(\R^N))}. \end{equation}
In other words, if $\Omega_n$ $\gamma$-converges to the empty set, then $\lambda_k(\Omega_n) \to +\infty$ for every $k \geq 1$.
Conversely, if $\Omega$ is a $s$-quasi open set such that $w_\Omega= 0$, then $(-\Delta)^s w_\Omega = 0$ in $\Omega$, and therefore $\Omega=\emptyset$ quasi-everywhere.
\end{remark}

\subsection{Weak $\gamma$-convergence}
Since $\mathcal{A}(\R^N)$ is not compact in the topology of $\gamma$-convergence, we introduce the notion of \emph{weak $\gamma$-convergence} for which $\mathcal{A}(\R^N)$ is sequentially compact.

In this section we prove that a functional $J$ defined in $\mathcal{A}(\R^N)$ which is   l.s.c. with respect to the $\gamma$-convergence is also l.s.c. with respect to the weak $\gamma$-convergence if it is assumed to be decreasing with respect to the inclusion of sets.

\begin{definition}
Let $\{\Omega_n\}_{n \in \N}$ be a sequence of $s$-quasi open sets. We say that $\{\Omega_n\}_{n \in \N}$ \emph{weakly $\gamma$-converges} to the $s$-quasi open set $\Omega$ if the solutions $w_n \in H^s(\R^N)$ of the problems
\begin{equation} \label{wa}
\left\{\begin{array}{r c l l} (-\Delta)^s w_{\Omega_n} & = & 1 & \text{in }\Omega_n, \\ w_{\Omega_n} & = & 0 & \text{in }\R^N \setminus \Omega_n,\end{array}\right.
\end{equation}
converge weakly in $H^s(\R^N)$, and strongly in $L^2(\R^N)$, to a function $w \in H^s(\R^N)$ such that $\Omega = \{w > 0\}$.
\end{definition}

\begin{prop} \label{semicontinuityvolume}
Let $\{\Omega_n\}_{n \in \N}$ be a sequence of $s$-quasi open sets of uniformly bounded measure, which weakly $\gamma$-converges to the $s$-quasi open set $\Omega$. Then,
\[ |\Omega| \leq \liminf_{n \to +\infty} |\Omega_n|.\]
\end{prop}
\begin{proof}
Let $m:= \liminf_{n \to +\infty} |\Omega_n|$. Up to extracting a subsequence, we can suppose that $m=\lim_{n \to +\infty} |\Omega_n|$. Let $w_{\Omega_n} \in H^s_0(\Omega_n)$ be the sequence of torsion functions defined  in \eqref{wa}. Since $w_{\Omega_n} \to w$ strongly in $L^2(\R^N)$, there exists a subsequence $w_{\Omega_{n_k}}$ such that $w_{\Omega_{n_k}}$ converges almost everywhere in $\R^N$ to $w$. Since $\Omega = \{w > 0\}$, it holds $\chi_{\Omega} \leq \liminf_{k \to +\infty} \chi_{\Omega_{n_k}}$ almost everywhere in $\R^N$. By Fatou's Lemma,
\[ |\Omega| = \int_{\R^N} \chi_\Omega \leq \liminf_{k \to +\infty }\int_{R^N} \chi_{\Omega_{n_k}} = m\]
as required.
\end{proof}

\begin{prop} \label{strongimpliesweak}
Let $\{\Omega_n\}_{n \in \N}$ be a sequence of $s$-quasi open sets   of uniformly bounded measure  which $\gamma$-converges to the $s$-quasi open set $\Omega$. Then $\{\Omega_n\}_{n \in \N}$ weakly $\gamma$-converges to $\Omega$.
\end{prop}
\begin{proof}
The proof can be performed as in \cite[Remark 4.7.8]{henrotpierre}.
\end{proof}

\begin{lema} \label{substitutelemma}
Suppose that $\{\Omega_n\}_{n \in \N}$ is a sequence of $s$-quasi open sets   of uniformly bounded measure  which weakly $\gamma$-converges to the $s$-quasi open set $\Omega$. Let $\{u_n\}_{n \in \N}$ be a sequence of functions in $H^s(\R^N)$ such that $u_n \in H^s_0(\Omega_n)$ for every $n \in \N$, and $u_n \rightharpoonup u$ weakly in $H^s(\R^N)$. Then, $u \in H^s_0(\Omega)$.
\end{lema}
\begin{proof}
The proof can be performed as in \cite[Lemma 4.7.10]{henrotpierre} 
\end{proof}

\begin{lema} \label{inscatolati}
Let $\{\Omega_n\}_{n \in \N}$ be a sequence of $s$-quasi open sets of uniformly bounded measure, which weakly $\gamma$-converges to the $s$-quasi open set $\Omega$. Then, there exists an increasing sequence of positive integers $\{n_k\}_{k \in \N}$ and a sequence of quasi-open sets $\{C_k\}_{k \in \N}$ such that $\Omega_{n_k} \subset C_k$ for every $k \in \N$, and $\{C_k\}_{k \in \N}$ $\gamma$-converges to $\Omega$. 
\end{lema}
\begin{proof}
The proof can be performed as in \cite[Lemma 4.7.11]{henrotpierre}, where Lemma \ref{substitutelemma} should be used instead of \cite[Lemma 4.7.10]{henrotpierre}. 
\end{proof}

Finally, we state the main result of this section.

\begin{prop} \label{similartoBDM}
Let $J:\mathcal{A}(\R^N) \to (-\infty, +\infty]$ be a functional satisfying:
\begin{enumerate}
 \item[(i)] $J$ is decreasing with respect to the inclusion of sets;
 \item[(ii)] $J$ is lower semicontinuous with respect to the $\gamma$-convergence.
\end{enumerate}
Then $J$ is lower semicontinuous with respect to the weak $\gamma$-convergence.
\end{prop}
\begin{proof}
Let $\{\Omega_n\}_{n \in \N}$ be a sequence of $s$-quasi open sets of uniformly bounded measure, which weakly $\gamma$-converges to the $s$-quasi open set $\Omega$. By Lemma \ref{inscatolati}, there exists an increasing sequence of positive integers $\{n_k\}_{k \in \N}$ and a sequence of quasi-open sets $\{C_k\}_{k \in \N}$ such that \[\lim_{n \to +\infty} J(\Omega_{n_k}) = \liminf_{n \to +\infty} J(\Omega_n),\] $\Omega_{n_k} \subset C_k$ for every $k \in \N$, and $\{C_k\}_{k \in \N}$ $\gamma$-converges to $\Omega$. Since $J$ is decreasing with respect to the inclusion of sets,
\[ J(\Omega) \leq \liminf_{k \to +\infty} J(C_k) \leq \liminf_{k \to +\infty} J(\Omega_{n_k}) = \liminf_{n \to +\infty} J(\Omega_n).\]
The proof is concluded.
\end{proof}

\section{Proof of Theorem \ref{main}} \label{sec.main}

In the following, $\{\Omega_n\}_{n\in\N}$ will be a sequence of $s$-quasi open sets of uniformly bounded measure. The proof of Theorem \ref{main}, which will be performed in several steps, is based on the behavior of the sequence $\{w_{\Omega_n}\}_{n\in\N}$ according to the concentration-compactness principle stated in Proposition \ref{ccp}.  Without loss of generality, we can suppose that $\int_{\R^N} |w_{\Omega_n}|^2 \to \lambda$ as $n \to +\infty$ for some $\lambda > 0$.

\subsection{Compactness for $w_{\Omega_n}$}

Assume that $\{w_{\Omega_n}\}_{n\in\N}$ is in the compactness case, that is, up to some subsequence still denoted with the same index, and some translations, the sequence $\{w_{\Omega_n}\}_{n\in\N}$ converges strongly in $L^2(\R^n)$ to some $w \in H^s(\R^N)$. Then, by definition, $\Omega_n$ weakly $\gamma$-converges to the set $\Omega:=\{w>0\}$. 

\subsection{Vanishing for $w_{\Omega_n}$}

In the spirit of \cite{lieb} we prove the following lemma.

\begin{lema}\label{lemma.lieb}
Let $A$ and $B$ be two measurable sets. Then there exists $z\in \R^N$ such that, if $A_z = z +A$,
$$
\lambda_1(A_z\cap B)\leq 2(\lambda_1(A) + \lambda_1(B)).
$$
\end{lema}

\begin{proof}  The roles of $u$ and $v$ were reversed, and also $x$ and $z$.
Let $z \in \R^N$ be arbitrary and let $u$ and $v$ be  positive first eigenfunctions on $A$ and $B$ respectively, normalized such that $\|u\|_{L^2(A)}=\|v\|_{L^2(B)}=1$. By regularity,  the function $u_z$ defined by $u_z(x)=u(z-x)$ satisfies  $u_z \in H^s_0(A_z) \cap L^\infty(A_z)$, and $v \in H^s_0(B) \cap L^\infty(B)$. The function $w_z$ defined as $w_z(x)=u(x-z)v(x)$ belongs to $H^s_0(A_z \cap B) \cap L^\infty(A_z \cap B)$. Define
\[ T(z):= \int_{\R^N} \int_{\R^N} \frac{|w_z(x)-w_z(y)|^2}{|x-y|^{N+2s}}\,dx\,dy, \qquad D(z):= \int_{\R^N} |w_z(x)|^2\,dx.\]
It holds that
\[ \int_{\R^N} D(z)\,dz = \int_{\R^N} \int_{\R^N} |w_z(x)|^2\,dx\,dz = \int_{\R^N} \int_{\R^N} |u(x-z)v(x)|^2\,dx\,dz = 1.\]
Moreover,
\begin{align*}  
& |w_z(x)-w_z(y)|^2 \\ & = |u(x-z)v(x) - u(y-z)v(y)|^2 \\   & =  |u(x-z)v(x) - u(x-z)v(y) + u(x-z)v(y)- u(y-z)v(y)|^2 \\  & =  |u(x-z)|^2 |v(x)-v(y)|^2 + |v(y)|^2 |u(x-z)-u(y-z)|^2 \\  & \; \;\;+ 2u(x-z)v(y)[v(x)-v(y)][u(x-z)-u(y-z)].\end{align*}
Using the elementary inequality $2ab\leq a^2+b^2$, the last term in the inequality above can be bounded as
$$
|u(x-z)|^2 |v(x)-v(y)|^2 + |v(y)|^2 |u(x-z)-u(y-z)|^2,
$$
and from the last two expressions we get
$$
|w_z(x)-w_z(y)|^2\leq 2\left( |u(x-z)|^2 |v(x)-v(y)|^2 +|v(y)|^2 |u(x-z)-u(y-z)|^2 \right).
$$
Thus
\begin{align*}
T(z) \leq 2\int_{\R^N}\int_{\R^N} \frac{|u(x-z)|^2 |v(x)-v(y)|^2}{|x-y|^{N+2s}}\,dx \,dy +2\int_{\R^N}\int_{\R^N} \frac{|v(y)|^2 |u(x-z)-u(y-z)|^2}{|x-y|^{N+2s}}\,dx\,dy.
\end{align*}
Then, integrating over $z$ and performing a change of variables, since $u$ and $v$ are normalized in $L^2$ norm, we get
$$
\int_{\R^N}T(z)\,dz \leq 2(\lambda_1(A)+\lambda_1(B)):=\Lambda.
$$
Therefore, $\int_{\R^N} [T(z)-\Lambda D(z)]\,dz \leq 0$, hence $0\leq T(z)\leq \Lambda D(z)$ on a set of positive measure. From the definitions of $T$, $D$ and $\Lambda$ the lemma follows.
\end{proof}

Assume that $\{w_{\Omega_n}\}_{n\in\N}$ is in the vanishing case, that is, for all $R>0$ it holds that
\[ 
 \lim_{n \to +\infty} \sup_{y \in \R^N} \int_{y + B_R} |w_{\Omega_n}|^2= 0.
\]
Since the sequence $\{w_{\Omega_n}\}_{n\in\N}\subset H^s_0(\R^N)$, we can assume that $w_{\Omega_n}\rightharpoonup w$ weakly in $H^s(\R^N)$.   Fix $\varepsilon>0$. By Lemma \ref{lemma.lieb}, there exists $R>0$ and a sequence $\{y_n\}_{n \in \N}$ in $\R^N$ such that
\begin{equation} \label{relationlieb} \lambda_1((y_n+\Omega_n) \cap B_{R})\leq 2\lambda_1(\Omega_n) + \varepsilon.\end{equation}
From the weak maximum principle  it follows that $w_{y_n+\Omega_n}\geq w_{(y_n+\Omega_n)\cap B_{R}}\geq 0$, and then, the vanishing assumption on $w_{\Omega_n}$ gives that 
\[ 
 \lim_{n \to +\infty}  \int_{B_{R}} |w_{(y_n+\Omega_n)\cap B_{R}}|^2 = 0.
\] 
This means that $w_{(y_n+\Omega_n)\cap B_{R}} \to 0$ strongly in $L^2(\R^N)$, and therefore $(y_n+\Omega_n)\cap B_{R}$ $\gamma-$converges to the empty set. By Remark \ref{remarkemptyset}, \[\lambda_1((y_n+\Omega_n)\cap B_{R})\to +\infty \qquad \text{as } n\to+\infty.\] By \eqref{relationlieb} we obtain that
$$
\lambda_1(\Omega_n)\to+\infty \quad \text{as } n \to +\infty.
$$  
From the Poincar\'e inequality given in Proposition \ref{poincare} we find that
$$
\|w_{\Omega_n}\|_{L^2(\Omega_n)}\leq \frac{1}{\lambda_1(\Omega_n)} [w_{\Omega_n}]_{H^s(\R^N)}  \to 0 \quad \text{as }n\to+\infty
$$
since $w_{\Omega_n}\in H^s_0(\R^N)$ is bounded in $H^s(\R^N)$.
Finally, by Proposition \ref{uniformconvergenceresolvent} and Remark \ref{remarkemptyset}  we obtain that $\|R_{\Omega_n}\|_{\mathcal{L}(L^2(\R^N))} \to 0$. By definition, the sequence $\{\Omega_n\}_{n \in \N}$ $\gamma$-converges, and hence weakly $\gamma$-converges, to the empty set.

\subsection{Dichotomy for $w_{\Omega_n}$} Finally, suppose that $w_{\Omega_n}$ is in the dichotomy case. That means that it is possible to find two sequences $\{u_n\}_{n \in \N}$ and $\{v_n\}_{n \in \N}$ of nonnegative functions in $H^s_0(\Omega_n)$   and a number $\alpha \in (0,\lambda)$  such that, up to a subsequence,
\[
\|w_{\Omega_n}-u_n-v_n\|_{L^2(\R^N)}  \to 0 \qquad \text{as }n \to +\infty;
\]
\[ 
\int_{\R^N} u_n^2 \to \alpha ,\qquad \int_{\R^N} v_n^2 \to \lambda-\alpha \qquad \text{for }n \to +\infty;
\]
\[ 
\dist(\supp u_n,\supp v_n) \to +\infty \qquad \text{for }n \to +\infty;
\]
\begin{equation} \label{liminfconccomp2bis} 
\liminf_{n \to +\infty} \left( [w_{\Omega_n}]^2_{H^s(\R^N)} - [u_n]^2_{H^s(\R^N)}-[v_n]^2_{H^s(\R^N)}\right) \geq 0. \end{equation}
We define the following sets
\begin{equation} \label{seqq}
 \Omega_n^1 := \{u_n>0\},\quad \Omega_n^2 :=\{v_n >0\},\quad \widetilde{\Omega}_n := \Omega_n^1 \cup \Omega_n^2, 
\end{equation}
and then $\widetilde{\Omega}_n$ is a quasi-open set contained in $\Omega_n$.

The proof of the claims in the dichotomy case will be a consequence of the following three lemmas.

\begin{lema}
The sequence of sets \eqref{seqq} satisfies \[\liminf_{n \to +\infty} |\Omega_n^i| > 0 \qquad \text{for }i=1,2.\]
\end{lema}
\begin{proof}
  Suppose by contradiction that, for instance, $\liminf_{n \to +\infty} |\Omega_n^1|=0$. The functions $w_{\Omega_n}$ are uniformly bounded in $L^\infty$ by \cite[Theorem 3.1]{brascoparini}, and therefore, by construction, also the functions $u_n$ are uniformly bounded in $L^\infty$. But then, $\int_{\R^N} u_n^2 \to 0$, which contradicts the fact that $\int_{\R^N} u_n^2 \to \alpha>0$.
\end{proof}

\begin{lema} \label{lema.a.tilde.a}
With the previous notation, we have that
\[ 
\|w_{\Omega_n} - w_{\widetilde{\Omega}_n}\|_{H^s(\R^N)} \to 0 \quad \text{ as }n\to+\infty.
\]
\end{lema}
\begin{proof}
We observe that $w_{\widetilde{\Omega}_n}$ is the orthogonal projection of $w_{\Omega_n}$ on the space $H^s_0(\widetilde{\Omega}_n)$. Indeed, let us consider the functional $F:H^s_0(\tilde\Omega_n)\to\R$ defined by 
\[
F(v)=\frac{1}{2} [w_{\Omega_n}-v]_{H^s(\R^N)}^2.
\]
Observe that
\[
F(v)=\frac{1}{2}[w_{\Omega_n}]_{H^s(\R^N)}^2 +\frac{1}{2} [v]_{H^s(\R^N)}^2 - \int_{\R^N}\int_{\R^N} \frac{(w_{\Omega_n}(x)-w_{\Omega_n}(y))(v(x)-v(y))}{|x-y|^{N+2s}}\,dxdy.
\]
Using the weak formulation of $w_{\Omega_n}$ we have that
\[
F(v)=\frac{1}{2}[w_{\Omega_n}]_{H^s(\R^N)}^2 +\frac{1}{2} [v]_{H^s(\R^N)}^2 - \int_{\widetilde{\Omega}_n} v.
\]
Then, the functional $F$ will be minimized for $v=w_{\tilde \Omega_n}$, since $w_{\tilde\Omega_n}$ minimizes the functional \[ v \mapsto \frac{1}{2} [v]_{H^s(\R^N)}^2 - \int_{\widetilde{\Omega}_n} v. \]
Hence,
\begin{align*}
 \int_{\R^N}\int_{\R^N} & \frac{|w_{\Omega_n}(x)-w_{\Omega_n}(y)-w_{\widetilde{\Omega}_n}(x)+w_{\widetilde{\Omega}_n}(y)|^2}{|x-y|^{N+2s}}\,dx\,dy  \\ 
 & \leq \int_{\R^N}\int_{\R^N} \frac{|w_{\Omega_n}(x)-w_{\Omega_n}(y)-(u_n+v_n)(x)+(u_n+v_n)(y))|^2}{|x-y|^{N+2s}}\,dx\,dy \\ 
 & = [w_{\Omega_n}]^2_{H^s(\R^N)} + [u_n+v_n]^2_{H^s(\R^N)} \\ & - 2\int_{\R^N}\int_{\R^N} \frac{[w_{\Omega_n}(x)-w_{\Omega_n}(y)][(u_n+v_n)(x)-(u_n+v_n)(y)]}{|x-y|^{N+2s}}\,dx\,dy  \\ 
 & = \int_{\R^N} w_{\Omega_n}  + [u_n+v_n]^2_{H^s(\R^N)} - 2\int_{\R^N} (u_n+v_n)  \\ 
 & = 2\left(\int_{\R^N} w_{\Omega_n}- \int_{\R^N} (u_n+v_n) \right)  +[u_n+v_n]^2_{H^s(\R^N)} - [w_{\Omega_n}]^2_{H^s(\R^N)}.
\end{align*}  
Observe that
\[ \bigg|\int_{\R^N} w_{\Omega_n} - \int_{\R^N} (u_n+v_n) \bigg| \leq |\Omega_n|^{\frac{1}{2}} \|w_{\Omega_n}-(u_n+v_n)\|_{L^2(\R^N)} \to 0\]
as $n \to  +\infty$. Moreover, using the fact that $[u_n + v_n]_{H^s(\R^N)}^2 \leq [u_n]_{H^s(\R^N)}^2 + [v_n]_{H^s(\R^N)}^2$ since they are nonnegative functions, we obtain from \eqref{liminfconccomp2bis} that
\[ \limsup_{n \to  +\infty} \left( [u_n+v_n]^2_{H^s(\R^N)} - [w_{\Omega_n}]^2_{H^s(\R^N)} \right) \leq 0\]
and therefore  
\[ 
[w_{\Omega_n} - w_{\widetilde{\Omega}_n}]_{H^s(\R^N)} \to 0 \quad \text{ as }n\to+\infty.
\]
By Proposition \ref{poincare}, there exists $C>0$ such that, for every $n \in \N$,
\[ \|w_{\Omega_n} - w_{\widetilde{\Omega}_n}\|_{L^2(\R^N)} = \|w_{\Omega_n} - w_{\widetilde{\Omega}_n}\|_{L^2(\Omega_n)} \leq C[w_{\Omega_n} - w_{\widetilde{\Omega}_n}]_{H^s(\R^N)}\]
and hence
\[ 
\|w_{\Omega_n} - w_{\widetilde{\Omega}_n}\|_{H^s(\R^N)} \to 0 \quad \text{ as }n\to+\infty.
\]
\end{proof}

\begin{lema} \label{lema.f}
Let $\tilde \Omega\subset \Omega \subset \R^N$ two sets of finite measure. There exists a constant $C=C(|\Omega|,N)>0$ and $\alpha>0$ such that
$$
\|R_\Omega -R_{\tilde \Omega} 	\|_{\mathcal{L}(L^2(\R^N))} \leq C \|w_\Omega - w_{\tilde \Omega} \|_{L^2(\R^N)}^\alpha.
$$
\end{lema}

\begin{proof}
Let $0<s<1$ be fixed. Observe that if $u,v \in H^1_0(\Omega)$ are the unique solutions of $(-\Delta)^s u=f$ in $\Omega$,  $(-\Delta)^s v=1$ in $\Omega$, respectively, using $v$ and $u$ as test functions in the weak formulation of the two previous equations, respectively, we get
$$
	\int_{\R^N}\int_{\R^N} \frac{(u(x)-u(y))(v(x)-v(y))}{|x-y|^{N+2s}} \,dx\,dy = \int_\Omega fw = \int_\Omega u,
$$
that is, $\int_\Omega f w_\Omega = \int_\Omega R(f)$. The previous computation gives that
\begin{equation*} 
\int_\Omega R_\Omega(f)-R_{\tilde \Omega}(f) =  \int_\Omega f(w_\Omega-w_{\tilde \Omega}).
\end{equation*}

By \cite[Theorem 3.1]{brascoparini}, for $N<4s$ we have
\begin{equation} \label{reg}
	\|R_\Omega(f)\|_{L^\infty(\Omega)} \leq C(N,|\Omega|) \|f\|_{L^2(\Omega)},
\end{equation}
and then, by using \eqref{reg} and H\"older's inequality we get
\begin{align*}
\|R_\Omega(f)-R_{\tilde \Omega}(f)\|_{L^2(\Omega)}^2 &\leq  \|R_\Omega(f)-R_{\tilde \Omega}(f)\|_{L^\infty(\Omega)}  \|R_\Omega(f)-R_{\tilde \Omega}(f)\|_{L^1(\Omega)}\\
&\leq C \|f\|_{L^2(\Omega)} \|f(w_\Omega-w_{\tilde \Omega})\|_{L^1(\Omega)}\\ 
&\leq C \|f\|_{L^2(\Omega)}^2 \|w_\Omega-w_{\tilde \Omega}\|_{L^2(\Omega)}.
\end{align*}

The case  $N\geq 4s$ will follow by an interpolation argument. For that end, consider $p>2$, $N\geq 4s$ and $f\in L^p(\Omega)$, $f\geq 0$. By using again \cite[Theorem 3.1]{brascoparini} and H\"older's inequality we get
$$
\|R_\Omega(f)-R_{\tilde \Omega}(f)\|_{L^p(\Omega)}  \leq C \|f\|_{L^p(\Omega)} \|w_\Omega-w_{\tilde \Omega}\|^\frac1p_{L^{p'}(\Omega)}
$$
for a suitable constant $C$ depending only on $p$, $N$ and $|\Omega|$, that is, 
$$
\|R_\Omega -R_{\tilde \Omega} \|_{\mathcal{L}(L^p(\R^N))} \leq C   \|w_\Omega-w_{\tilde \Omega}\|^\frac1p_{L^{p'}(\Omega)}.
$$
Now, 
let  $R^*_\Omega$ and $R^*_{\tilde \Omega}$ be the adjoint operators  of $R_\Omega$ and $\Omega_{\tilde \Omega}$, respectively, which are defined from $L^{p'}(\Omega)$ in itself. Since the $L^{p'}$ norm of $R^*_\Omega-R^*_{\tilde \Omega}$ coincides with the $L^p$ norm of $R_\Omega-R_{\tilde \Omega}$, we get
$$
\|R^*_\Omega -R^*_{\tilde \Omega} \|_{\mathcal{L}(L^{p'}(\R^N))}  \leq C  \|w_\Omega-w_{\tilde \Omega}\|^\frac1p_{L^{p'}(\Omega)}.
$$
	Since $R_\Omega$ and $R_{\tilde \Omega}$ are self-adjoint on $L^2(\Omega)$, keeping the same notation for $R_A$, $R_{\tilde \Omega}$ and their extension on $L^{p'}(\Omega)$, we obtain that $R_\Omega - R_{\tilde \Omega}: L^{p'}(\Omega)  \to L^{p'}(\Omega)$ and 
$$
\|R_\Omega -R_{\tilde \Omega} \|_{\mathcal{L}(L^{p'}(\R^N))} \leq C   \|w_\Omega-w_{\tilde \Omega}\|^\frac1p_{L^{p'}(\Omega)}.
$$
Finally, from the Riesz-Thorin interpolation theorem and since $1<p'<2$, we obtain that
\begin{align*}
\|R_\Omega -R_{\tilde \Omega} \|_{\mathcal{L}(L^2(\R^N))} &\leq \|R_\Omega -R_{\tilde \Omega} \|_{\mathcal{L}(L^p(\R^N))}^\frac12 \|R_\Omega -R_{\tilde \Omega} \|_{\mathcal{L}(L^{p'}(\R^N))}^\frac12\\
&\leq C  \|w_\Omega-w_{\tilde \Omega}\|_{L^{p'}(\Omega)}^\frac1p\\
&\leq C  |\Omega|^\frac{2-p'}{p'p} \|w_\Omega-w_{\tilde \Omega}\|_{L^2(\Omega)}^\frac1p
\end{align*}
which ends the proof.

\end{proof}

\section{Proof of Theorem \ref{compactnessdichotomy}} \label{sec.compactnessdichotomy}

Let $\{\Omega_n\}_{n \in \N} \subset \mathcal{A}(\R^N)$ be a minimizing sequence for Problem \eqref{minimizingproblem}, satisfying $|\Omega_n|=c$ for every $n \in \N$, and $J(\Omega_n) \to m$ as $n \to +\infty$. By Theorem \ref{main}, we have two possible cases:
\begin{enumerate}
 \item[(i)] there exists a subsequence, still denoted by $\{\Omega_n\}_{n \in \N}$, and a set $\Omega \in \mathcal{A}(\R^N)$, such that, up to some translations, $\{\Omega_n\}_{n \in \N}$ weakly $\gamma$-converges to $\Omega$. Since $J$ is invariant by translations, the sequence will be again a minimizing sequence for $J$. By Proposition \ref{semicontinuityvolume}, $|\Omega|\leq c$. Let $\hat{\Omega} \in \mathcal{A}(\R^N)$ be such that $\Omega \subset \hat{\Omega}$ and $|\hat{\Omega}|=c$. Since $J$ is decreasing with respect to set inclusion, and by Propositions \ref{uniformconvergenceresolvent} and \ref{similartoBDM},
\[ m \leq J(\hat{\Omega}) \leq J(\Omega) \leq \liminf_{n \to +\infty} J(\Omega_n) = m.\]
Therefore, $\hat{\Omega}$ is a minimizing set.
 \item[(ii)] there exists a subsequence, still denoted by $\{\Omega_n\}_{n \in \N}$, such that we can define $\widetilde{\Omega}_n = \Omega_n^1 \cup \Omega_n^2 \subset \Omega_n$, where $\Omega_n^1$, $\Omega_n^2$ are such that $\text{dist}(\Omega_n^1,\Omega_n^2) \to +\infty$, $\liminf_{n \to +\infty} |\Omega_n^i| > 0$ for $i=1,2$, and $J(\widetilde{\Omega}_n) \to m$ as $n \to +\infty$. If $|\widetilde{\Omega}_n|<c$, it is possible to modify suitably the sequence in order to respect the volume constraint as well, since the functional $J$ is decreasing with respect to set inclusion.
\end{enumerate}

\end{document}